\documentclass[12pt]{amsart}
\usepackage[utf8]{inputenc}
\usepackage{amsfonts}
\usepackage{amsmath}
\usepackage{amssymb}
\usepackage{amsthm}
\usepackage{xcolor}
\usepackage[margin=1.3in]{geometry}
\usepackage{graphicx}
\usepackage[11pt]{moresize}

\newtheorem{theorem}{Theorem}
\newtheorem*{conjecture*}{Conjecture}
\newtheorem{proposition}{Proposition}
\newtheorem{lemma}{Lemma}

\theoremstyle{remark}

\newtheorem*{remark*}{Remark}

\newcommand{\vs}{\vspace{3mm}}
\newcommand{\vsss}{\vspace{6mm}}
\newcommand{\un}{\underline}

\newcommand{\Z}{\mathbb{Z}}

\newcommand{\Prob}{\mathbb{P}}

\newcommand{\ep}{\epsilon}
\newcommand{\lam}{\lambda}

\newcommand{\wt}{\widetilde}

\title{New Lower Bounds for Trace Reconstruction}
\author{Zachary Chase}
\address{Mathematical Institute, Andrew Wiles Building, Radcliffe Observatory Quarter, Woodstock Road, Oxford OX2 6GG, UK}
\email{zachary.chase@maths.ox.ac.uk}
\date{July 23, 2020}

\begin{document}

\begin{abstract}
We improve the lower bound on worst case trace reconstruction from $\Omega\left(\frac{n^{5/4}}{\sqrt{\log n}}\right)$ to $\Omega\left(\frac{n^{3/2}}{\log^{7} n}\right)$. As a consequence, we improve the lower bound on average case trace reconstruction from $\Omega\left(\frac{\log^{9/4}n}{\sqrt{\log\log n}}\right)$ to $\Omega\left(\frac{\log^{5/2}n}{(\log\log n)^{7}}\right)$.
\end{abstract}

\maketitle

\section{Introduction}

Given a string $x \in \{0,1\}^n$, a \textit{trace} of $x$ is obtained by deleting each bit of $x$ with probability $q$, independently, and concatenating the remaining string. For example, a trace of $11001$ could be $101$, obtained by deleting bits $2$ and $3$. The goal of the trace reconstruction problem is to determine an unknown string $x$, with high probability, by looking at as few independently generated traces of $x$ as possible. 

\vs

More precisely, fix $\delta,q \in (0,1)$. Take $n$ large. For each $x \in \{0,1\}^n$, let $\mu_x$ be the probability distribution on $\{0,1\}^{\le n}$ given by $\mu_x(w) = (1-q)^{|w|}q^{n-|w|}f(w;x)$, where $f(w;x)$ is the number of times $w$ appears as a subsequence in $x$, that is, the number of strictly increasing tuples $(i_1,\dots,i_{|w|})$ such that $x_{i_j} = w_j$ for $1 \le j \le |w|$. The problem is to determine the minimum value of $T = T(n)$ for which there exists a function $f: (\{0,1\}^{\le n})^T \to \{0,1\}^n$ satisfying $\Prob_{\mu_x^T}[f(\wt{U}^1,\dots,\wt{U}^T) = x] \ge 1-\delta$ for each $x \in \{0,1\}^n$ (where the $\wt{U}^j$ denote the $T$ independently generated traces).

\vs

The problem of trace reconstruction was introduced by Batu, Kannan, Khanna, and McGregor [1] as ``an abstraction and simplification of a fundamental problem in bioinformatics, where one desires to reconstruct a common ancestor of several organisms given genetic sequences from those organisms." [2]

\vs

Holenstein, Mitzenmacher, Panigrahy, and Wieder [3] established an upper bound, that $\exp(\wt{O}(n^{1/2}))$ traces suffice. Nazarov and Peres [4] and De, O'Donnell, and Servedio [5] simultaneously obtained the best upper bound known, that $\exp(O(n^{1/3}))$ traces suffice. The lower bound of $\Omega(n)$ was established in [1], by considering the strings $0^{\frac{n}{2}-1}10^{\frac{n}{2}}$ and $0^{\frac{n}{2}}10^{\frac{n}{2}-1}$. Holden and Lyons [2] obtained the (previous) best lower bound known, by presenting two strings $x'_n \not = y'_n \in \{0,1\}^n$ which require $\Omega(n^{5/4}/\sqrt{\log n})$ traces to distinguish between. Their idea was to keep a 1 as a ``defect" in the middle of the string, but to ``pad" with $01$'s instead of $0$'s.

In this paper, we improve the lower bound, exhibiting two strings $x_n \not = y_n \in \{0,1\}^n$ which require $\Omega(n^{3/2}/\log^{7} n)$ traces to distinguish between. In fact, our methods show that $\Omega(n^{3/2}/\log^{7} n)$ traces are required to distinguish between $x'_n$ and $y'_n$ as well (a (messier) analogue of \eqref{eqn:explicit} holds). We also use the idea of padding a ``defect" 1 with $01$'s. We chose strings slightly different than those considered in [2] for computational ease.

\vs

\noindent Let $k \ge 1$, $n = 4k+3$, and $x_n = (01)^k 1 (01)^{k+1}, y_n = (01)^{k+1}1(01)^k$, i.e.  

\hspace{30mm} $x_n = 0101...0101 \hspace{4mm} 1 \hspace{4mm} 01 \hspace{3.3mm} 0101...0101$

\hspace{30mm} $y_n = \hspace{.33mm} 0101...0101 \hspace{3mm} 01 \hspace{4mm} 1 \hspace{4.35mm} 0101...0101.$

\vspace{1mm}

\begin{theorem}
Fix $q, \delta \in (0,1)$. Then there exists some constant $c = c(q,\delta) > 0$ so that at least $cn^{3/2}/\log^{7}n$ traces are required to distinguish between $x_n$ and $y_n$ with probability at least $1-\delta$, under trace reconstruction with deletion probability $q$.
\end{theorem}

\vspace{1mm}

The main reason we are able to obtain an improvement over $n^{5/4}$ is that we explicitly compute (an upper bound for) the quantity relevant to determining the number of samples needed, rather than relying on a coupling argument to determine only the total variation distance of the measures induced on subsequences.

\vs

A variant of the trace reconstruction problem is, instead of being required to reconstruct any string $x$ from traces of it, one must reconstruct a string $x$ chosen uniformly at random from traces of it. For a formal statement of the problem, see Section 1.2 of [2]. The best upper bound known, due to Holden, Pemantle, and Peres, is that $\exp(O(\log^{1/3}n))$ traces suffice [6]. The (previous) best lower bound known was $\Omega(\frac{\log^{9/4}n}{\sqrt{\log\log n}})$ [2]. Proposition 4.1 of [2] together with Theorem 1 implies

\vs

\begin{theorem}
For all $q \in (0,1)$, there is $c = c(q) > 0$ so that for all large $n$, the probability of reconstructing a random $n$-bit string from $c\log^{5/2}(n)/(\log\log n)^{7}$ traces is at most $\exp(-n^{0.15})$, under trace reconstruction with deletion probability $q$.
\end{theorem} 

\vs

Very recently, other variants of the trace reconstruction problem have been considered. The interested reader should refer to [7], [8], [9], and [10]. 

\vs

Here is an outline of the paper. In Section $2$, we recall ``the distance" (namely, the Hellinger distance) between two probability measures that is directly relevant for determining the number of samples needed to distinguish between them, and we deduce Theorem 1 assuming an appropriate estimate. In Section $3$, we prove the estimate by obtaining closed form expressions for the probability distributions induced by the traces of $x_n$ and $y_n$ and related expressions. In Section $4$, we give the proofs of some lemmas used throughout Section $3$. Finally, in Section $5$ we establish a result of independent interest, a nontrivial bound on the number of traces that suffice to distinguish between any pair of strings with a very large Hamming distance (in contrast to the small Hamming distance pair considered to get Theorem 1). 

\vspace{1.5mm}

\section{A Warmup to the Proof of Theorem 1}

Throughout the proof, $A \lesssim B$ means $A \le CB$ for some absolute constant $C$, and $A \asymp B$ means $A \lesssim B$ and $B \lesssim A$. We take $q = 1/2$ for ease; the (analogous) proof works for any $q \in (0,1)$. The variables (to be introduced later) $j,t,a,b,f,m$ will always be integers, the variables $\ep_j,\ep_t$ will always be integer multiples of $\frac{1}{3}$, and all expressions occurring in binomial coefficients will be integers (we clearly state when it appears otherwise due to slight abuse of notation). For a string $w$, we let $|w|$ denote the length of $w$, and for any positive integers $a,b$ with $a \le b$, we denote by $w_{a,b}$ the contiguous substring $w_a,w_{a+1},\dots,w_b$. 

\vs

Fix $n \equiv 3 \pmod{4}$ large. Let $k = \frac{n-3}{4}$. Let $\mu$ be the probability measure for the traces of $x_n$ and $\nu$ be the probability measure for the traces of $y_n$. Let $E$ be a subset of $\cup_{0 \le k \le n} \{0,1\}^k$ with $\mu(E),\nu(E) \ge 1-O(e^{-\frac{1}{2}\log^2 n})$. We define $E$ in Section 3.2. 

\vs

It is well known, though seemingly folklore, that the number of samples needed to distinguish between two probability distributions with high probability is proportional to the inverse square of the Hellinger distance between them (see, e.g., Lemma A.5 of [2]): $$\frac{1}{H(\mu,\nu)^2} = \frac{1}{\sum_w (\sqrt{\mu(w)}-\sqrt{\nu(w)})^2}.$$ Note $$\sum_w (\sqrt{\mu(w)}-\sqrt{\nu(w)})^2 \le \sum_{w \in E} (\sqrt{\mu(w)}-\sqrt{\nu(w)})^2+\sum_{w \not \in E} (\mu(w)+\nu(w)),$$ so since $$\mu(E^c),\nu(E^c) \le O(e^{-\frac{1}{2}\log^2 n}),$$ to show that $\Omega(\frac{n^{3/2}}{\log^7 n})$ traces are necessary to distinguish between $x_n$ and $y_n$, it suffices to show that $$\sum_{w \in E} (\sqrt{\mu(w)}-\sqrt{\nu(w)})^2 \lesssim \frac{\log^7 n}{n^{3/2}}.$$ And since $$\sum_{w \in E} (\sqrt{\mu(w)}-\sqrt{\nu(w)})^2 \le \sum_{w \in E} \left[\sqrt{\frac{\mu(w)}{\nu(w)}}+1\right]^2(\sqrt{\mu(w)}-\sqrt{\nu(w)})^2 = \sum_{w \in E} \frac{(\mu(w)-\nu(w))^2}{\nu(w)},$$ to prove Theorem 1, it suffices to show 

\begin{equation}\label{eqn:main}
\sum_{w \in E} \frac{(\mu(w)-\nu(w))^2}{\nu(w)} \lesssim \frac{\log^7 n}{n^{3/2}}.
\end{equation}

\vs

\section{Proving Inequality \eqref{eqn:main}}

\subsection{Obtaining Closed Form Expressions for $\mu$ and $\nu$\nopunct} \text{}

\vs

\noindent In this subsection, we obtain closed form expressions for the probability distributions of the traces of $x_n$ and $y_n$. Let $s_k = (01)^k =  0101\dots 01$ be of length $2k$. Let $f_{\text{c}}(w)$ denote the number of \un{c}ontiguous $01$ appearances in $w$.

\vs

We will use the following simple and fortuitous combinatorial lemma. It is the main reason we are able to obtain a simple(r) closed form expression.

\vs

\begin{lemma} For strings $w,z$, let $f(w;z)$ denote the number of times $w$ appears as a subsequence in $z$, that is, the number of strictly increasing tuples $(i_1,\dots,i_{|w|})$ such that $z_{i_j} = w_j$ for $1 \le j \le |w|$. Then, for any $k \ge 0$, $f(w;s_k) = {k+f_{\text{c}}(w) \choose m}$ if $|w| = m$.
\end{lemma}

\begin{proof}
The idea is that every $01$ occurring in $w$ is a chance to put two consecutive indices in $w$ in the same pair in $s_k$. Take any $1 \le j_1 < j_2 < \dots < j_m \le k+f_{\text{c}}(w)$. Let $I_1 = j_1$ and $I_{p+1} = I_p+j_{p+1}-j_p-1_{w_p = 0 = w_{p+1}-1}$ for $1 \le p \le m-1$. For each $1 \le p \le m$, let $i_p \in \{2I_p-1,2I_p\}$ be such that $w_p = (s_k)_{i_p}$. We thus get an occurrence of $w$ in $s_k$; conversely, given any occurrence of $w$ in $s_k$ via $(i_p)_{1 \le p \le m}$, we optain $(I_p)_{1 \le p \le m}$ and then $(j_p)_{1 \le p \le m}$ as above. The correspondence between $(j_p)_p$ and $(i_p)_p$ is a bijective one.
\end{proof}

\vs

\noindent Doing casework on whether $w$ includes the ``lone $1$" (i.e. the 1 at index $2k+1$ in $x$, and the 1 at index $2k+3$ in $y$, where the convention is that the first index is $1$), and if so, where it appears, Lemma 1 implies that \begin{equation}\label{eqn:explicitmu} 2^n\mu(w) = {2k+f_{\text{c}}(w) \choose |w|}+\sum_{\substack{1 \le j \le |w| \\ w_j = 1}} {k+f_{\text{c}}(w_{1,j-1}) \choose j-1}{k+1+f_{\text{c}}(w_{j+1,m}) \choose m-j}\end{equation} \begin{equation}\label{eqn:explicitnu} 2^n\nu(w) \hspace{1.3mm} = {2k+f_{\text{c}}(w) \choose |w|} + \sum_{\substack{1 \le j \le |w| \\ w_j = 1}} {k+1+f_{\text{c}}(w_{1,j-1}) \choose j-1}{k+f_{\text{c}}(w_{j+1,m}) \choose m-j}.\end{equation}

\vs

\subsection{The ``High Probability" Set $E$\nopunct} \text{}

\vs

\noindent We now define the ``high probability" set used in Section 2. Let $$E = \{w \in \{0,1\}^{\le n} : \big||w| - 2k\big| \le \sqrt{k}\log(k) \text{ and }\big|f_{\text{c}}(w)- \frac{2k}{3}\big| \le \sqrt{k}\log(k)\}.$$ 

\noindent In this subsection, we show $\mu(E), \nu(E) \ge 1-O(e^{-\frac{1}{2}\log^2 n})$. To this end, and for the purposes of proving inequality \eqref{eqn:main}, we make frequent use of the following technical lemma, used to estimate binomial coefficients. It is proven in Section 4.

\begin{lemma}
For any real $\eta$ bounded away from $0$ and $1$, any positive integers $A$ and $B$ such that $\eta A, \eta B \in \Z$, and any integers $\Delta$ and $\sigma$ such that $A+\Delta,\eta A+\sigma, B-\Delta$, and $\eta B-\sigma$ are non-negative, it holds that {\tiny $$\left[\frac{{A+\Delta \choose \eta A+\sigma}{B-\Delta \choose \eta B-\sigma}}{{A \choose \eta A}{B \choose \eta B}}\right]^{-1} =$$ $$(1+O(\frac{\sigma^3}{A^2}))(1+O(\frac{\Delta^3}{A^2}))(1+O(\frac{1}{A}))(1+O(\frac{\sigma(\Delta-\sigma)^2}{A^2}))(1+O(\frac{\Delta(\Delta-\sigma)^2}{A^2}))\exp\left(\frac{1}{2}\frac{(\Delta-\sigma)^2}{(1-\eta)A}+\frac{1}{2}\frac{\sigma^2}{\eta A}-\frac{1}{2}\frac{\Delta^2}{A}\right)$$ $$\times (1+O(\frac{\sigma^3}{B^2}))(1+O(\frac{\Delta^3}{B^2}))(1+O(\frac{1}{B}))(1+O(\frac{\sigma(\Delta-\sigma)^2}{B^2}))(1+O(\frac{\Delta(\Delta-\sigma)^2}{B^2}))\exp\left(\frac{1}{2}\frac{(\Delta-\sigma)^2}{(1-\eta)B}+\frac{1}{2}\frac{\sigma^2}{\eta B}-\frac{1}{2}\frac{\Delta^2}{B}\right).$$}
\end{lemma}

\vs

A corollary of Lemma 2 we will use frequently is that, if $A \le B$, say, then the product ${A+\Delta \choose \eta A +\sigma}{B-\Delta \choose \eta B-\sigma}$ is, up to a $(1+O(\frac{\log^3 A}{A}))$ multiplicative error, maximized at $\sigma = \Delta = 0$.

\vs

Formally, for any $\eta, A,B, \Delta$, and $\sigma$ with restrictions as in Lemma 2, we have \begin{equation}\label{eqn:corlem2} {A+\Delta \choose \eta A+\sigma} {B-\Delta \choose \eta B-\sigma} \lesssim {A \choose \eta A}{B \choose \eta B}.\end{equation}

\vs

\noindent For instance, \eqref{eqn:corlem2}, together with \eqref{eqn:explicitmu}, implies that for any $w \in E$, if $m := |w|$ and $f := f_{\text{c}}(w)$,\footnote{By $\frac{m}{2}$ and $\frac{f}{2}$, we mean $\lfloor \frac{m}{2} \rfloor$ and $\lfloor \frac{f}{2} \rfloor$. Similarly in the rest of the paper when $\frac{m}{2}$ and $\frac{f}{2}$ appear in binomial coefficients.} $$2^n\mu(w) \le {2k+f \choose m}+m\max_j {k+f_{\text{c}}(w_{1,j-1}) \choose j-1}{k+1+f-f_{\text{c}}(w_{j+1,m}) \choose m-j}$$ $$\hspace{-18mm} \le {2k+f \choose m}+m\max_{j,a} {k+a \choose j-1}{k+1+f-a \choose m-j}$$ $$\hspace{-56mm} \lesssim {2k+f \choose m} + m{k+\frac{f}{2} \choose \frac{m}{2}}^2$$ \begin{equation}\label{eqn:muwbound} \hspace{-76mm} \lesssim \sqrt{k}{2k+f \choose m}.\end{equation}

\vs

\noindent The following is another simple combinatorial lemma. 

\begin{lemma}
For positive integers $a$ and $l$, the number of $w \in \{0,1\}^l$ such that $f_{\text{c}}(w) = a$ is ${l+1 \choose 2a+1}$.
\end{lemma}

\begin{proof}
The number of such strings is equal to the number of ways to place $2a+1$ indistinguishable flags in $l+1$ spots. Indeed, any such string $w = (w_1,\dots,w_l)$ has exactly $2a+1$ indices $i$ (a ``flag"), $0 \le i \le l$ such that $w_i \not = w_{i+1}$, where we define $w_0 = 1$ and $w_{l+1} = 0$. And any choice of $2a+1$ flags corresponds to a $w$. This correspondence is a bijective one. 
\end{proof}

\vs

\noindent Continuing from \eqref{eqn:muwbound}, Lemma $3$ implies \begin{equation}\label{eqn:firstmu} \mu(\{w \in \{0,1\}^m : f_{\text{c}}(w) = f\}) \lesssim 2^{-n}\sqrt{k}{2k+f \choose m}{m+1 \choose 2f+1}.\end{equation} We now argue that we can restrict to $m$ close to $2k$, allowing us to use Lemma 2 to then show that the right side of \eqref{eqn:firstmu} is small for $f$ far from $\frac{2k}{3}$. Since, for any $m$, $\sum_{w \in \{0,1\}^m} \mu(w) = 2^{-n}{n \choose m}$ and since $2^{-n}{n \choose m} = O(e^{-\log^2 n})$ for $m \not \in [\frac{n}{2}-\sqrt{n}\log(n), \frac{n}{2}+\sqrt{n}\log(n)]$ (by, e.g., Lemma 2), we have \begin{equation}\label{eqn:apriori} \mu\left(\bigcup_{m \not \in [2k-\sqrt{k}\log(k),2k+\sqrt{k}\log(k)]} \{0,1\}^m\right) = O(e^{-\frac{1}{2}\log^2 n}).\end{equation} Now assume $|m-2k| \le \sqrt{k}\log(k)$. Writing $m = 2k+\delta$ and $f = \frac{2k}{3}+\ep$, we see that $${2k+f \choose m}{m \choose 2f} = {\frac{8k}{3}+\ep \choose 2k+\delta}{2k+\delta \choose \frac{4k}{3}+2\ep}$$ $$\hspace{39.5mm} = {\frac{8k}{3}+\ep \choose \frac{4k}{3}+2\ep}{\frac{4k}{3}-\ep \choose \frac{2k}{3}-2\ep+\delta}.$$

\noindent Continuing from \eqref{eqn:firstmu}, using Lemma 2 with $A = \frac{8k}{3}, B = \frac{4k}{3}, \Delta = \epsilon, \sigma = 2\ep-\frac{2\delta}{3}$, and 

\noindent $\eta = \frac{2k+\delta}{4k} = \frac{1}{2}+O(\frac{\log k}{\sqrt{k}})$, we see that $|f-\frac{2k}{3}| > \sqrt{k}\log(k)$ implies $$\mu(\{w \in \{0,1\}^m : f_{\text{c}}(w) = f\}) \lesssim 2^{-4k}\sqrt{k}e^{-\log^2 n} {8k/3 \choose 4k/3} {4k/3 \choose 2k/3}$$ $$ \hspace{12mm} \lesssim e^{-\log^2 n}.$$ Hence, since there are at most $n^2$ values of $(m,f)$, it holds that \begin{equation}\label{eqn:Arightm} \mu\left(\bigcup_{\substack{m \in [2k-\sqrt{k}\log(k),2k+\sqrt{k}\log(k)] \\ f \not \in [\frac{2k}{3}-\sqrt{k}\log(k),\frac{2k}{3}+\sqrt{k}\log(k)]}} \{w \in \{0,1\}^m : f_{\text{c}}(w) = f\}\right) = O(e^{-\frac{1}{2}\log^2 n}).\end{equation} Combining \eqref{eqn:apriori} and \eqref{eqn:Arightm}, we see \begin{equation}\label{eqn:boundAmu} \mu(E) \ge 1-O(e^{-\frac{1}{2}\log^2 n}).\end{equation} The same argument shows that \begin{equation}\label{eqn:boundAnu} \nu(E) \ge 1-O(e^{-\frac{1}{2}\log^2 n}).\end{equation}

\vs

\noindent We take a moment to prove the following lemma, useful in the upcoming two sections, which allows us to focus on the probablistically relevant ranges of the parameters involved. 

\begin{lemma}
Let $f$ and $m$ be positive integers such that $|f-\frac{2k}{3}|,|m-2k| \le \sqrt{k}\log(k)$. Then, for any positive integers $a,j$, it holds that ${k+a \choose j-1}{k+1+f-a \choose m-j} \lesssim e^{-\log^2 k}{\frac{4k}{3} \choose m/2}^2$ unless $|a-\frac{f}{2}| \le \sqrt{k}\log(k)$ and $|j-\frac{m}{2}| \le \sqrt{k}\log(k)$.
\end{lemma} 

\begin{proof}
Lemma 2 implies, for any $\lam,\beta = O(A^{1/6})$ and $\eta$ bounded away from $0$ and $1$, $${A+\lam\sqrt{A} \choose \eta A +\beta\sqrt{A}}{A-\lam\sqrt{A} \choose \eta A-\beta\sqrt{A}} \lesssim e^{\lam^2-\beta^2/\eta-(\lam-\beta)^2/(1-\eta)} {A \choose \eta A}{A \choose \eta A}.$$ We use $A = \lfloor k+\frac{f}{2} \rfloor, \eta = \frac{m/2}{k+\frac{f}{2}} = \frac{3}{4}+O(\frac{\log k}{\sqrt{k}}), \lam = \frac{a-\frac{f}{2}}{\sqrt{k+\frac{f}{2}}}$, and $\beta = \frac{j-\frac{m}{2}}{\sqrt{k+\frac{f}{2}}}$.
\end{proof}

\vs

\subsection{A Closed Form Expression\nopunct} \text{}

\vs

\noindent In this subsection, we obtain a closed form expression for an upper bound of $\sum_{w \in E} \frac{(\mu(w)-\nu(w))^2}{\nu(w)}$, up to an acceptable (for the purposes of proving \eqref{eqn:main}) error. By the definition of $E$ and an obvious lower bound on $\nu$ coming from \eqref{eqn:explicitnu}, we have {\footnotesize \begin{equation}\label{eqn:hellgone}\sum_{w \in E} \frac{(\mu(w)-\nu(w))^2}{\nu(w)} \le \sum_{\substack{m \in [2k-\sqrt{k}\log(k),2k+\sqrt{k}\log(k)] \\ f \in [\frac{2k}{3}-\sqrt{k}\log(k),\frac{2k}{3}+\sqrt{k}\log(k)]}} \frac{1}{2^n{2k+f \choose m}}\sum_{\substack{|w| = m \\f_{\text{c}}(w) = f}} (2^n\mu(w)-2^n\nu(w))^2.\end{equation}}

\vs

\noindent We fix $m$ and $f$ and focus on estimating {\scriptsize $$\sum_{\substack{|w| = m \\f_{\text{c}}(w) = f}} \left(2^n\mu(w)-2^n\nu(w)\right)^2 = $$ $$\sum_{\substack{|w| = m \\f_{\text{c}}(w) = f}} \left(\sum_{1 \le j \le m : w_j = 1} {k+f_{\text{c}}(w_{1,j-1}) \choose j-1}{k+1+f_{\text{c}}(w_{j+1,m}) \choose m-j}-{k+1+f_{\text{c}}(w_{1,j-1}) \choose j-1}{k+f_{\text{c}}(w_{j+1,m}) \choose m-j}\right)^2 $$ \begin{equation}\label{eqn:explicit} = \sum_{1 \le j,t \le m} \sum_{\substack{|w| = m \\ f_c(w) = f \\ w_j = 1 \\ w_t = 1}} \left[{k+f_{\text{c}}(w_{1,j-1}) \choose j-1}{k+1+f_{\text{c}}(w_{j+1,m}) \choose m-j}-{k+1+f_{\text{c}}(w_{1,j-1}) \choose j-1}{k+f_{\text{c}}(w_{j+1,m}) \choose m-j}\right]$$ $$\hspace{25.15mm} \times \left[{k+f_{\text{c}}(w_{1,t-1}) \choose t-1}{k+1+f_{\text{c}}(w_{t+1,m}) \choose m-t}-{k+1+f_{\text{c}}(w_{1,t-1}) \choose t-1}{k+f_{\text{c}}(w_{t+1,m}) \choose m-t}\right], \end{equation}} 

\vspace{2mm}

\noindent where \eqref{eqn:explicit} refers to the expression occupying the final two lines. The first equality follows from \eqref{eqn:explicitmu} and \eqref{eqn:explicitnu}, and the second follows by expanding out the square and interchanging summations. 

\vs

We take the following page and a half to make restrictions on the variables involved in \eqref{eqn:explicit}, allowing us to make future estimates more effectively. 

\vs

We may restrict \eqref{eqn:explicit} to $j,t \in [\frac{m}{2}-\sqrt{k}\log(k),\frac{m}{2}+\sqrt{k}\log(k)]$ and $w$ with $|f_c(w_{1,j-1})-\frac{f}{2}| \le \sqrt{k}\log(k)$ and $|f_c(w_{1,t-1})-\frac{f}{2}| \le \sqrt{k}\log(k)$. Indeed, if at least one of those four restrictions does not hold, then by Lemma 4 and \eqref{eqn:corlem2}, {\scriptsize $${k+f_c(w_{1,j-1}) \choose j-1}{k+f_c(w_{j+1,m}) \choose m-j}{k+f_c(w_{1,t-1}) \choose t-1}{k+f_c(w_{t+1,m}) \choose m-t} \lesssim e^{-\log^2 k} {\frac{4k}{3} \choose m/2}^2.\footnote{The fact that this bound is (more than) sufficient to indeed make the said restrictions follows from the same argument, about to be made, yielding \eqref{eqn:upcoming}.}$$ }

\vspace{1.5mm}

A quick calculation shows that {\scriptsize $${k+f_c(w_{1,j-1}) \choose j-1}{k+1+f_c(w_{j+1,m}) \choose m-j} - {k+1+f_c(w_{1,j-1}) \choose j-1}{k+f_c(w_{j+1,m}) \choose m-j}$$ $$ = {k+f_c(w_{1,j-1}) \choose j-1}{k+f_c(w_{j+1,m}) \choose m-j}\left[\frac{m-j}{k+1+f_c(w_{j+1,m}) - (m-j)}-\frac{j-1}{k+1+f_c(w_{1,j-1})-(j-1)}\right].$$} 

\vspace{1.5mm}

The restrictions just made ensure that \begin{equation}\label{eqn:sqrtbound} \frac{m-j}{k+1+f_c(w_{j+1,m}) - (m-j)}-\frac{j-1}{k+1+f_c(w_{1,j-1})-(j-1)} = O\left(\frac{\log(k)}{\sqrt{k}}\right).\end{equation} Indeed, since $k+1+f_c(w_{j+1,m})-(m-j) \ge \frac{k}{3}-O(\sqrt{k}\log(k))$ and $k+1+f_c(w_{1,j-1})-(j-1) \ge \frac{k}{3}-O(\sqrt{k}\log(k))$, we have $$\frac{m-j}{k+1+f_c(w_{j+1,m}) - (m-j)}-\frac{j-1}{k+1+f_c(w_{1,j-1})-(j-1)}$$ $$ \hspace{34mm} = \frac{m-j}{k+f_c(w_{j+1,m}) - (m-j)}-\frac{j}{k+f_c(w_{1,j-1})-j}+O(\frac{1}{k})$$ $$ \hspace{31mm} = \frac{mk-2jk-jf_c(w_{j+1,m})+(m-j)f_c(w_{1,j-1})}{(k+f_c(w_{j+1,m})-m+j)(k+f_c(w_{1,j-1})-j)}+O(\frac{1}{k})$$ $$ \hspace{-37.5mm}  = \frac{O(k\sqrt{k}\log(k))}{\Omega(k^2)}$$ $$ \hspace{-40.3mm} = O\left(\frac{\log(k)}{\sqrt{k}}\right).$$

\vs

Up to a multiplicative factor of $2$, we may restrict \eqref{eqn:explicit} to $t > j$ (the argument about to be made shows the diagonal $t=j$ term is sufficiently small). Furthermore, we may in fact restrict to $t > j+5$; indeed, by \eqref{eqn:corlem2}, Lemma 3, and \eqref{eqn:sqrtbound}, we see that expression \eqref{eqn:explicit} with the first sum restricted to $j < t \le j+5$ is upper bounded by $$5\sum_{j \in [k-\sqrt{k}\log(k),k+\sqrt{k}\log(k)]}\sum_{\substack{|w| = m \\ f_{\text{c}}(w) = f}} {k+\frac{f}{2} \choose \frac{m}{2}}^4 \frac{\log^2(k)}{k} \lesssim \frac{\log^3(k)}{\sqrt{k}}{k+\frac{f}{2} \choose \frac{m}{2}}^4 {m \choose 2f},$$

\noindent and so summing this over $|m-2k| \le \sqrt{k}\log(k)$ and $|f-\frac{2k}{3}| \le \sqrt{k}\log(k)$ with weights $\frac{1}{2^n{2k+f \choose m}}$, we obtain an upper bound up to a multiplicative constant for {\ssmall $$\sum_{\substack{|m-2k| \le \sqrt{k}\log(k) \\ |f-\frac{2k}{3}| \le \sqrt{k}\log(k)}} \frac{1}{2^n{2k+f \choose m}}\sum_{\substack{1 \le j < t \le m \\ t \le j+5}} \sum_{\substack{|w| = m \\ f_{\text{c}}(w) = f \\ w_j = 1, w_t = 1}}\left[{k+f_{\text{c}}(w_{1,j-1}) \choose j-1}{k+1+f_{\text{c}}(w_{j+1,m}) \choose m-j}-{k+1+f_{\text{c}}(w_{1,j-1}) \choose j-1}{k+f_{\text{c}}(w_{j+1,m}) \choose m-j}\right]$$ $$\hspace{55mm} \times \left[{k+f_{\text{c}}(w_{1,t-1}) \choose t-1}{k+1+f_{\text{c}}(w_{t+1,m}) \choose m-t}-{k+1+f_{\text{c}}(w_{1,t-1}) \choose t-1}{k+f_{\text{c}}(w_{t+1,m}) \choose m-t}\right]$$} of \begin{equation}\label{eqn:upcoming} (\sqrt{k}\log(k))^2\sup_{\substack{|m-2k| \le \sqrt{k}\log(k) \\ |f-\frac{2k}{3}| \le \sqrt{k}\log(k)}} \frac{\log^3(k)}{\sqrt{k}}\frac{{k+\frac{f}{2} \choose \frac{m}{2}}^2}{{2k+f \choose m}}\frac{{k+\frac{f}{2} \choose \frac{m}{2}}^2{m \choose 2f}}{2^{4k}} \lesssim \frac{\log^5(k)}{k^{3/2}} \lesssim \frac{\log^5(n)}{n^{3/2}}.\end{equation}

\noindent One should compare to \eqref{eqn:hellgone}, the equation involving \eqref{eqn:explicit}, and \eqref{eqn:main}. 

\vspace{5mm}

The following very important paragraph, which ignores multiplicative constants, explains the motivation behind the rest of the calculations in this paper.

\vs

In the calculations just above, we used the trivial upper bound of ${k+\frac{f}{2} \choose \frac{m}{2}}^4\frac{\log^2(k)}{k}$ for the summands of \eqref{eqn:explicit}. If we did not restrict to $t \le j+5$ in the calculation just above and used that same trivial upper bound (which is indeed valid for $j,t \in [\frac{m}{2}-\sqrt{k}\log(k),\frac{m}{2}+\sqrt{k}\log(k)]$), we would get an upper bound for the right hand side of \eqref{eqn:hellgone} of $\frac{\log^5(n)}{n^{3/2}}\sqrt{n}\log(n) = \frac{\log^6(n)}{n}$, since there are $\sqrt{k}\log(k)$ values of $t$ rather than just $5$. Therefore, we just need a savings of $\sqrt{k}/\log(k)$ over that trivial upper bound to obtain \eqref{eqn:main}. Note in that trivial upper bound, we just bounded each term individually, not using any cancellation amongst the different summands. Our goal in Section 3.4 is to analyze the left hand side of \eqref{eqn:sqrtbound} very carefully, in order to exploit cancellation between different summands of \eqref{eqn:explicit}. To make the paper significantly shorter, we do not repeatedly make the type of calculation just made above; rather, we point out where $\widetilde{\Omega}(\sqrt{k})$ savings come from as we go along. 

\vspace{5mm}

\noindent Fix some $t$ and $j$ with $t > j+5.$\footnote{We wanted to restrict to $t > j+5$ so that the following case analysis has no ``boundary issues".} We will now separate the sum over $w$ in \eqref{eqn:explicit} based on $f_{\text{c}}(w_{1,j-1})$ and $f_{\text{c}}(w_{1,t-1})$. To relate $f_{\text{c}}(w_{1,j-1})$ to $f_{\text{c}}(w_{j+1,m})$ and $f_{\text{c}}(w_{1,t-1})$ to $f_{\text{c}}(w_{t+1,m})$ given $f_{\text{c}}(w)$, we need to do casework on $w_{j-1}$ and $w_{t-1}$. We first do the case of $w_{j-1} = w_{t-1} = 0$. In this case, $f_{\text{c}}(w_{j+1,m}) = f-f_{\text{c}}(w_{1,j-1})-1$ and $f_{\text{c}}(w_{t+1,m}) = f-f_{\text{c}}(w_{1,t-1})-1$. This gives the ``first case" of \eqref{eqn:explicit}: {\scriptsize $$\sum_{\substack{|w| = m \\ f_{\text{c}}(w) = f \\ w_{j-1} = 0, w_j = 1 \\ w_{t-1} = 0, w_t = 1}} \left[{k+f_{\text{c}}(w_{1,j-1}) \choose j-1}{k+1+f_{\text{c}}(w_{j+1,m}) \choose m-j}-{k+1+f_{\text{c}}(w_{1,j-1}) \choose j-1}{k+f_{\text{c}}(w_{j+1,m}) \choose m-j}\right]$$ $$\hspace{14mm} \times \left[{k+f_{\text{c}}(w_{1,t-1}) \choose t-1}{k+1+f_{\text{c}}(w_{t+1,m}) \choose m-t}-{k+1+f_{\text{c}}(w_{1,t-1}) \choose t-1}{k+f_{\text{c}}(w_{t+1,m}) \choose m-t}\right]$$} $$= \sum_{a,b \ge 0} \sum_{\substack{|w| = m \\ f_{\text{c}}(w) = f \\ w_{j-1}=0, w_j = 1 \\ w_{t-1} = 0, w_t = 1 \\ f_{\text{c}}(w_{1,j-1}) = a \\ f_{\text{c}}(w_{1,t-1}) = b}} \left[{k+a \choose j-1}{k+f-a \choose m-j}-{k+1+a \choose j-1}{k+f-a-1 \choose m-j}\right]$$ $$\hspace{29mm} \times \left[{k+b \choose t-1}{k+f-b \choose m-t}-{k+1+b \choose t-1}{k+f-b-1 \choose m-t}\right].$$

\noindent Removing the product (that does not depend on $w$) from the inner sum, we wish to count the set of $w$ with $|w| = m, f_{\text{c}}(w) = f, w_{j-1} = 0, w_j = 1, w_{t-1} = 0, w_t = 1, f_{\text{c}}(w_{1,j-1}) = a$, and $f_{\text{c}}(w_{1,t-1}) = b$. Noting that $f_{\text{c}}(w_{1,j-1}) = f_{\text{c}}(w_{1,j-2})$, we use $$f_{\text{c}}(w_{1,t-1}) = f_{\text{c}}(w_{1,j-1})+f_{\text{c}}(w_{j-1,t-1}) = f_{\text{c}}(w_{1,j-1})+1+f_{\text{c}}(w_{j+1,t-1})$$ and $$f_{\text{c}}(w_{j+1,t-1}) = f_{\text{c}}(w_{j+1,t-2})$$ 

\noindent together with Lemma 3 to get that the number of such $w$ is ${j-1 \choose 2a+1}{t-j-1 \choose 2b-2a-1}{m-t+1 \choose 2f-2b-1}$. So, the case of $w_{j-1} = w_{t-1} = 0$ yields expression \eqref{eqn:explicitcase1}: {\scriptsize \begin{equation}\label{eqn:explicitcase1} \sum_{\substack{t > j+5 \\ a,b \ge 0}} \left[{k+a \choose j-1}{k+f-a \choose m-j}-{k+1+a \choose j-1}{k+f-a-1 \choose m-j}\right] {j-1 \choose 2a+1}{t-j-1 \choose 2b-2a-1}{m-t+1 \choose 2f-2b-1}\end{equation} $$
\indent \times \left[{k+b \choose t-1}{k+f-b \choose m-t}-{k+1+b \choose t-1}{k+f-b-1 \choose m-t}\right].$$}

\noindent The other three cases of the value of the pair $(w_{j-1},w_{t-1})$ yield very similar expressions. The only difference between the expressions is that some binomial coefficients have $-1, -2, +1,+2$, or $0$ in certain places. However, these minor differences will not affect our proceeding arguments. That is, our argument for a $\sqrt{k}/\log(k)$ savings for the $(w_{j-1},w_{t-1}) = (0,0)$ case would show a $\sqrt{k}/\log(k)$ savings for the other 3 cases. Therefore, we may restrict attention to the case $(w_{j-1},w_{t-1}) = (0,0)$.

\vs

\subsection{Finishing the Proof of \eqref{eqn:main}\nopunct} \text{}

\vs

\noindent In this final subsection, we appropriately bound \eqref{eqn:explicitcase1}, thereby proving \eqref{eqn:main}. As explained in the last section, we may assume $a \in [\frac{f}{2}-\sqrt{k}\log(k),\frac{f}{2}+\sqrt{k}\log(k)]$, thereby, as before, yielding $${k+a \choose j-1}{k+f-a \choose m-j}-{k+1+a \choose j-1}{k+f-a-1 \choose m-j} =$$ $${k+a \choose j-1}{k+f-a-1 \choose m-j}\left[\frac{m-j}{k+f-a-(m-j)}-\frac{j}{k+a-j}+O\left(\frac{1}{k}\right)\right].$$ Let $\delta_j$ and $\ep_j$ be defined so that $$j = \frac{m}{2}+\delta_j$$ and $$a = \frac{j}{3}+\frac{f}{2}-\frac{m}{6}+\ep_j.$$ Observe that $$\frac{m-j}{k+f-a-(m-j)}-\frac{j}{k+a-j} = \frac{-2k\delta_j+\frac{m\delta_j}{3}+m\ep_j-f\delta_j}{(k+f-a-\frac{m}{2}+\delta_j)(k+a-\frac{m}{2}-\delta_j)}.$$

\noindent Since $a \in [\frac{f}{2}-\sqrt{k}\log(k),\frac{f}{2}+\sqrt{k}\log(k)]$, we have $\ep_j = O(\sqrt{k}\log(k))$. Since also $m = 2k+O(\sqrt{k}\log(k))$ and $f = \frac{2k}{3}+O(\sqrt{k}\log(k))$, we see that $$\frac{m-j}{k+f-a-(m-j)}-\frac{j}{k+a-j} = 18\frac{1}{k}[\ep_j-\delta_j]+O\left(\frac{\log^2(k)}{k}\right).$$ Therefore, defining $\delta_t$\footnote{We are abusing notation here. Formally, define a function $\delta$ by $\delta(x) = x-\frac{m}{2}$; we use $\delta_j$ as shorthand for $\delta(j)$ and $\delta_t$ as shorthand for $\delta(t)$. Analogously for $\ep_j,\ep_t$.} and $\ep_t$ so that $$t = \frac{m}{2}+\delta_t$$ and $$b = \frac{t}{3}+\frac{f}{2}-\frac{m}{6}+\ep_t,$$ we see that \eqref{eqn:explicitcase1} takes the form \begin{equation} \label{eqn:newexplicitcase1} \frac{324}{k^2}\sum_{a,b,t,j} {k+a \choose j-1}{k+f-a-1 \choose m-j}{j-1 \choose 2a+1}{t-j-1 \choose 2b-2a-1}{m-t+1 \choose 2f-2b-1}\end{equation} $$ \times {k+b \choose t-1}{k+f-b-1 \choose m-t}\left[\delta_j-\ep_j\right]\cdot\left[\delta_t-\ep_t\right]$$ up to an acceptable error (the error is acceptable since it replaces a bound of $\frac{\log(k)}{\sqrt{k}}$ for $|\delta_j-\ep_j|$, say, with $\frac{\log^2(k)}{k}$, giving our desired $\log(k)/\sqrt{k}$ savings). Recall that we are summing over $t,j \in [k-\sqrt{k}\log(k),k+\sqrt{k}\log(k)]$.

\vs

\noindent We now claim that, unless $b = a+\frac{t-j}{3}+O(\sqrt{t-j}\log(k))$, the magintude of the summand corresponding to $a,b,j,t$ is sufficiently small. Note that ${t-j-1 \choose 2b-2a-1}{m-t+1 \choose 2f-2b-1} \asymp {\delta_t-\delta_j \choose \frac{2}{3}(\delta_t-\delta_j)+2\ep_t-2\ep_j}{\frac{m}{2}-\delta_t \choose f-\frac{2\delta_t}{3}-2\ep_t}$. We may use Lemma 2 with $A = \delta_t-\delta_j, B = \frac{m}{2}-\delta_t, \eta = \frac{f-\frac{2}{3}\delta_j-2\ep_j}{\frac{m}{2}-\delta_j} = \frac{2}{3}+O(\frac{\log k}{\sqrt{k}}), \Delta = 0, \sigma = 2\ep_t-2\ep_j-\left(\frac{f-\frac{2}{3}\delta_j-2\ep_j}{\frac{m}{2}-\delta_j}-\frac{2}{3}\right)(\delta_t-\delta_j)$ to deduce that $(b-a-\frac{t-j}{3})^2 > (t-j)\log^2(k)$ implies an $e^{-\log^2(n)}$ savings, verifying the claim.

\vs

\vs

\noindent Lemma 2 also implies that\footnote{Technically, we are adding and substracting $\lfloor a+\frac{t-j}{3} \rfloor$ rather than $a+\frac{t-j}{3}$.} {\small $${k+b \choose t-1}{k+f-b-1 \choose m-t} = {k+a+\frac{t-j}{3} \choose t-1}{k+f-a-\frac{t-j}{3}-1 \choose m-t}\left(1+O\left(\frac{\log^5(k)}{\sqrt{k}}\right)\right)$$} 

\noindent for $b = a+\frac{t-j}{3}+O(k^{1/4}\log^{3/2}(k))$. Therefore, we see that \eqref{eqn:newexplicitcase1} is, up to a multiplicative factor of $1+O(\frac{\log^5(k)}{\sqrt{k}})$, equal to \begin{equation}\label{eqn:newerexplicitcase1} \frac{324}{k^2}\sum_{a,b,j,t} {k+a \choose j-1}{k+f-a-1 \choose m-j}{j-1 \choose 2a+1}{t-j-1 \choose 2b-2a-1}{m-t+1 \choose 2f-2b-1}\end{equation} $$\indent \times {k+a+\frac{t-j}{3} \choose t-1}{k+f-a-\frac{t-j}{3}-1 \choose m-t}\left[\delta_j-\ep_j\right]\cdot\left[\delta_t-\ep_t\right],$$ where the sum is restricted to $|b-a-\frac{t-j}{3}| \le \sqrt{t-j}\log(k)$. 

\vs

Our strategy now to exploit cancellation occurring between different summands is as follows. We split the term $\delta_t-\ep_t$ into three terms and deal with each separately, each by fixing $j,t$, and $a$, and summing over $b$. We get cancellation from the second term by pairing the summand corresponding to $b$ to the summand corresponding to the reflection of $b$ about a natural symmetry (explained below). The third term has magnitude a factor of $\sqrt{k}$ less than $\delta_t-\ep_t$ (i.e. it is $O(1)$), so it can be ignored. The first term requires the most work and is dealt with after the second and third are handled. 

\vs

\noindent Specifically, we split up $$\left[\delta_t-\ep_t\right] = \left[\delta_t-\ep_j\right]+\left[\ep_j+\left(\frac{f-\frac{2}{3}\delta_j-2\ep_j}{m-2\delta_j}-\frac{1}{3}\right)(\delta_t-\delta_j)-\ep_t\right]$$ $$-\left[\left(\frac{f-\frac{2}{3}\delta_j-2\ep_j}{m-2\delta_j}-\frac{1}{3}\right)(\delta_t-\delta_j)\right].$$

\noindent For any fixed $a,j,$ and $t$, by Lemma 2 with $A = \delta_t-\delta_j, B = \frac{m}{2}-\delta_t, \eta = \frac{f-\frac{2}{3}\delta_j-2\ep_j}{\frac{m}{2}-\delta_j} = \frac{2}{3}+O(\frac{\log k}{\sqrt{k}}), \Delta = 0, \sigma = 2\ep_t-2\ep_j-\left(\frac{f-\frac{2}{3}\delta_j-2\ep_j}{\frac{m}{2}-\delta_j}-\frac{2}{3}\right)(\delta_t-\delta_j)$, we have that $${\delta_t-\delta_j \choose \frac{2}{3}(\delta_t-\delta_j)+2\ep_t-2\ep_j}{\frac{m}{2}-\delta_t \choose f-\frac{2}{3}\delta_t-2\ep_t} = $$ $$\left(1+O\left(\frac{\log^3(k)}{\sqrt{\delta_t-\delta_j}}\right)\right){\delta_t-\delta_j \choose \frac{2}{3}(\delta_t-\delta_j)+2\ep_t^*-2\ep_j} {\frac{m}{2}-\delta_j \choose f-\frac{2}{3}\delta_t-2\ep_t^*},$$ 

\noindent where $\ep_t^*$ is the reflection\footnote{To be precise, the reflection of $x$ about $y$ is defined to be $2y-x$.} of $\ep_t$ about $\ep_j+\frac{1}{2}(\frac{f-\frac{2}{3}\delta_j-2\ep_j}{\frac{m}{2}-\delta_j}-\frac{2}{3})(\delta_t-\delta_j)$.\footnote{We might have to round $\ep_t^*$ a bit (so that $\frac{t}{3}+\frac{f}{2}-\frac{m}{6}+\ep_t^*$ is an integer), but the induced error in this rounding is negligible, by Lemma 2.} And therefore, since $$\ep_j+\frac{1}{2}\left(\frac{f-\frac{2}{3}\delta_j-2\ep_j}{\frac{m}{2}-\delta_j}-\frac{2}{3}\right)(\delta_t-\delta_j)-\ep_t = O\left(\sqrt{\delta_t-\delta_j}\log(k)+\log^2(k) \right),$$ letting $b^*$ denote the $b$ corresponding to $\ep_t^*$, we deduce that $$\hspace{-10mm} \left|\sum_b {t-j-1 \choose 2b-2a-1}{m-t+1 \choose 2f-2b-1}\left[\ep_j+\left(\frac{f-\frac{2}{3}\delta_j-2\ep_j}{m-2\delta_j}-\frac{1}{3}\right)(\delta_t-\delta_j)-\ep_t\right]\right|$$ $$= \frac{1}{2} \bigg|\sum_b {t-j-1 \choose 2b-2a-1}{m-t+1 \choose 2f-2b-1}\left[\ep_j+\left(\frac{f-\frac{2}{3}\delta_j-2\ep_j}{m-2\delta_j}-\frac{1}{3}\right)(\delta_t-\delta_j)-\ep_t\right]$$ $$\hspace{15mm} + {t-j-1 \choose 2b^*-2a-1}{m-t+1 \choose 2f-2b^*-1}\left[\ep_j+\left(\frac{f-\frac{2}{3}\delta_j-2\ep_j}{m-2\delta_j}-\frac{1}{3}\right)(\delta_t-\delta_j)-\ep_t^*\right]\bigg|$$ \vspace{.2mm} $$= \frac{1}{2} \bigg|\sum_b {t-j-1 \choose 2b-2a-1}{m-t+1 \choose 2f-2b-1}\left[\ep_j+\left(\frac{f-\frac{2}{3}\delta_j-2\ep_j}{m-2\delta_j}-\frac{1}{3}\right)(\delta_t-\delta_j)-\ep_t\right]$$ $$\hspace{5mm} - \text{{\ssmall $\left(1+O\left(\frac{\log^3(k)}{\sqrt{\delta_t-\delta_j}}\right)\right)$}} {t-j-1 \choose 2b-2a-1}{m-t+1 \choose 2f-2b-1}\left[\ep_j+\left(\frac{f-\frac{2}{3}\delta_j-2\ep_j}{m-2\delta_j}-\frac{1}{3}\right)(\delta_t-\delta_j)-\ep_t\right]\bigg|$$ $$\hspace{2mm} \lesssim \sum_b {t-j-1 \choose 2b-2a-1}{m-t+1 \choose 2f-2b-1}O\left(\frac{\log^3(k)}{\sqrt{\delta_t-\delta_j}}\right)O\left(\sqrt{\delta_t-\delta_j}\log(k)+\log^2(k)\right)$$ $$\hspace{-67mm} \lesssim \sum_b {t-j-1 \choose 2b-2a-1}{m-t+1 \choose 2f-2b-1}\log^5(k)$$

\noindent is small enough, since we rid of a factor of $\widetilde{\Omega}(\sqrt{k})$ potentially coming from $\delta_t-\ep_t$. And since $(\frac{f-\frac{2}{3}\delta_j-2\ep_j}{m-2\delta_j}-\frac{1}{3})(\delta_t-\delta_j) = O(1)$ rather than $\Omega(\sqrt{k})$, the expression corresponding to the second term, namely $$\sum_b {t-j-1 \choose 2b-2a-1}{m-t+1 \choose 2f-2b-1}\left[\left(\frac{f-\frac{2}{3}\delta_j-2\ep_j}{m-2\delta_j}-\frac{1}{3}\right)(\delta_t-\delta_j)\right],$$ is small enough. Therefore, for any fixed $a,j,t$, the part of the sum in \eqref{eqn:newerexplicitcase1} with terms containing $b$ is, up to negligible error, the expression corresponding to the remaining term: \begin{equation}\label{eqn:newestexplicitcase1} \sum_b {t-j-1 \choose 2b-2a-1}{m-t+1 \choose 2f-2b-1}\left[\delta_t-\ep_j\right].\end{equation}

\noindent If $t > j+5$, Lemma 5, proven in Section 4, states that {\small$$\sum_b {t-j-1 \choose 2b-2a-1}{m-t+1 \choose 2f-2b-1} = \left(\frac{1}{2}+O\left(\frac{\log^2(k)}{t-j}\right)\right) \sum_b {t-j-1 \choose b-2a-1}{m-t+1 \choose 2f-b-1}.$$} 
\noindent And using the general combinatorial identity $$\sum_C {D \choose C}{E \choose F-C} = {D+E \choose F}$$ (we may extend the range of $b$ and restrict it freely, since the $b$ outside $a+\frac{t-j}{3}\pm\sqrt{t-j}\log(k)$ yield exponentially (in $\log^2 n$) small terms), we see that \begin{equation}\label{eqn:byeb} \sum_b {t-j-1 \choose 2b-2a-1}{m-t+1 \choose 2f-2b-1} = \left(\frac{1}{2}+O\left(\frac{\log^2 k}{t-j}\right)\right){m-j \choose 2f-2a-2}.\end{equation}

\noindent Therefore, noting that $\delta_t-\ep_j$ does not depend on $b$ and then plugging \eqref{eqn:byeb} into \eqref{eqn:newestexplicitcase1}, we see that \eqref{eqn:newerexplicitcase1} is, up to a negligible error, equal to {\small \begin{equation}\label{eqn:finalexplicitcase1} \frac{162}{k^2}\sum_{a,j,t} \left(1+O\left(\frac{\log^2(k)}{t-j}\right)\right){k+a \choose j-1}{k+f-a-1 \choose m-j}{j-1 \choose 2a+1}{m-j \choose 2f-2a-2}\end{equation} $$\times {k+a+\frac{t-j}{3} \choose t-1}{k+f-a-\frac{t-j}{3}-1 \choose m-t}\left[\delta_j-\ep_j\right]\cdot [\delta_t-\ep_j],$$} where, to reiterate, the sum is restricted to $t > j+5$.

\vs

\noindent We can rid of the $O(\frac{\log^2(k)}{t-j})$ term trivially. Indeed, using \eqref{eqn:corlem2}, we can upper bound $${k+a \choose j-1}{k+f-a-1 \choose m-j} \lesssim {k+\frac{f}{2} \choose \frac{m}{2}}^2,$$ $${k+a+\frac{t-j}{3} \choose t-1}{k+f-a-\frac{t-j}{3}-1 \choose m-t} \lesssim {k+\frac{f}{2} \choose \frac{m}{2}}^2,$$ and $${j-1 \choose 2a+1}{m-j \choose 2f-2a-2} \lesssim {\frac{m}{2} \choose f}^2;$$ noting that for each $\Delta \ge 5$, the number of pairs $(t,j) \in [\frac{m}{2}-\sqrt{k}\log(k),\frac{m}{2}+\sqrt{k}\log(k)]$ with $t-j = \Delta$ is at most $\sqrt{k}\log(k)$, we thus obtain an upper bound of $$\frac{162}{k^2}\sqrt{k}\log(k)\sqrt{k}\log(k){k+\frac{f}{2} \choose \frac{m}{2}}^4{\frac{m}{2} \choose f}^2\sum_{\Delta = 5}^{\sqrt{k}\log(k)} \frac{\log^2(k)}{\Delta},$$ which is small enough; i.e., $t-j$ is on average $\sqrt{k}$, which gives us the required savings (note we get the $\log^7(n)$ from here, since summing over $m$ and $f$ picks up two extra $\log(k)$ factors). Note that we needed the error in Lemma 5 to be $O(\frac{\log^2(k)}{t-j})$ rather than the trivial $O(\frac{\log^2(k)}{\sqrt{t-j}})$, since the latter would have led to the sum $\sum_{\Delta=5}^{\sqrt{k}\log(k)} \frac{\log^2(k)}{\sqrt{\Delta}}$, which would have yielded a $k^{1/4}$ factor rather than a $\log(k)$ factor. 

\vs

\noindent Let $$f(\delta_j,\ep_j) = \frac{162}{k^2}{k+a \choose j}{k+f-a-1 \choose m-j}{j-1 \choose 2a+1}{m-j \choose 2f-2a-2}$$ and $$g(\delta_t,\ep_j) = {k+a+\frac{t-j}{3} \choose t-1}{k+f-a-\frac{t-j}{3}-1 \choose m-t}.$$

\vs

\noindent We break up the remaining expression, i.e., expression \eqref{eqn:finalexplicitcase1} without the $O(\frac{\log^2(k)}{t-j})$ term, as follows: {\scriptsize$$\sum_{\substack{\ep_j,\delta_j,\delta_t \\ \delta_t > \delta_j+5}} f(\ep_j,\delta_j)[\delta_j-\ep_j]g(\delta_t,\ep_j)[\delta_t-\ep_j] = $$ \begin{equation}\label{eqn:finalerexplicitcase1} \sum_{\ep_j} \sum_{\delta_j > \ep_j} \left[f(\ep_j,\delta_j)(\delta_j-\ep_j)\sum_{\delta_t > \delta_j+5} g(\delta_t,\ep_j)(\delta_t-\ep_j)+f(\ep_j,2\ep_j-\delta_j)(\ep_j-\delta_j)\sum_{\delta_t > 2\ep_j-\delta_j+5} g(\delta_t,\ep_j)(\delta_t-\ep_j)\right].\end{equation}}

\vs

\noindent We claim that $g$ has symmetry\footnote{See footnote 7 on page 13.} in $\delta_t$ about $\ep_j$ and $f$ has symmetry in $\delta_j$ about $\ep_j$: $g(\delta_t,\ep_j) \approx g(2\ep_j-\delta_t,\ep_j)$ and $f(\ep_j,\delta_j) \approx f(\ep_j,2\ep_j-\delta_j)$. This is the content of the quite fortuitous Lemmas 6 and 7, respectively.\footnote{The additive factors of $-1,+1$, and $-2$ have been omitted for ease. The proofs are the same with them present.}

\vs

\setcounter{lemma}{5}
\begin{lemma} 
For any positive integers $f$ and $m$ with $|f-\frac{2k}{3}|, |m-2k| \le \sqrt{k}\log k$ and for any integers $\delta_t,\ep_j$ with $|\delta_t|,|\ep_j| \le \sqrt{k}\log k$ and $\frac{\delta_t}{3}+\ep_j \in \Z$, it holds that {\scriptsize$${k+\frac{f}{2}+\frac{\delta_t}{3}+\ep_j \choose \frac{m}{2}+\delta_t}{k+\frac{f}{2}-\frac{\delta_t}{3}-\ep_j \choose \frac{m}{2}-\delta_t} = \left(1+O\left(\frac{\log^3(k)}{\sqrt{k}}\right)\right) {k+\frac{f}{2}+\frac{\delta_t}{3}-\frac{5\ep_j}{3} \choose \frac{m}{2}+\delta_t-2\ep_j}{k+\frac{f}{2}-\frac{\delta_t}{3}+\frac{5\ep_j}{3} \choose \frac{m}{2}-\delta_t+2\ep_j}.$$}
\end{lemma}

\begin{proof}
Lemma 2, with $A = k+\frac{f}{2}, B = k+\frac{f}{2}, \eta = \frac{m/2}{k+\frac{f}{2}} = \frac{3}{4}+O(\frac{\log k}{\sqrt{k}}), \Delta = \frac{\delta_t}{3}+\ep_j, \sigma = \delta_t$ and $\Delta = \frac{\delta_t}{3}-\frac{5\ep_j}{3}, \sigma = \delta_t-2\ep_j$ shows that both products of binomial coefficients are $(1+O(\frac{\log^3(k)}{\sqrt{k}})) \exp(-\frac{3(\delta_t-\ep_j)^2}{k+\frac{f}{2}}){k+f/2 \choose m/2}^2$.
\end{proof}

\begin{lemma} 
For any positive integers $f$ and $m$ with $|f-\frac{2k}{3}|, |m-2k| \le \sqrt{k}\log k$ and for any integers $\delta_j,\ep_j$ with $|\delta_j|,|\ep_j| \le \sqrt{k}\log k$ and $\frac{\delta_j}{3}+\ep_j \in \Z$, it holds that {\scriptsize$${k+\frac{f}{2}+\frac{\delta_j}{3}+\ep_j \choose \frac{m}{2}+\delta_j}{k+\frac{f}{2}-\frac{\delta_j}{3}-\ep_j \choose \frac{m}{2}-\delta_j} = \left(1+O\left(\frac{\log^3(k)}{\sqrt{k}}\right)\right) {k+\frac{f}{2}+\frac{\delta_j}{3}-\frac{5\ep_j}{3} \choose \frac{m}{2}+\delta_j-2\ep_j}{k+\frac{f}{2}-\frac{\delta_j}{3}+\frac{5\ep_j}{3} \choose \frac{m}{2}-\delta_j+2\ep_j}$$} and {\scriptsize$${\frac{m}{2}+\delta_j \choose f+\frac{2\delta_j}{3}+2\ep_j}{\frac{m}{2}-\delta_j \choose f-\frac{2\delta_j}{3}-2\ep_j} = \left(1+O\left(\frac{\log^3(k)}{\sqrt{k}}\right)\right) {\frac{m}{2}+2\ep_j-\delta_j \choose f+\frac{10\ep_j}{3}-\frac{2\delta_j}{3}}{\frac{m}{2}-2\ep_j+\delta_j \choose f-\frac{10\ep_j}{3}+\frac{2\delta_j}{3}}.$$}
\end{lemma}

\begin{proof}
The first approximation is the content of Lemma 6. For the second, use Lemma 2 with $A = \frac{m}{2}, B = \frac{m}{2}, \eta = \frac{2f}{m} = \frac{2}{3}+O(\frac{\log k}{\sqrt{k}}), \Delta = \delta_j, \sigma = \frac{2\delta_j}{3}+2\ep_j$ and $\Delta = 2\ep_j-\delta_j, \sigma = \frac{10\ep_j}{3}-\frac{2\delta_j}{3}$ to see that both products of binomial coefficients are $(1+O(\frac{\log^3(k)}{\sqrt{k}})) \exp(-\frac{18\ep_j^2}{m/2}){\frac{m}{2} \choose f}^2$.
\end{proof}

\vs

\noindent Lemma 6 implies that, for each fixed $\delta_j$ and $\ep_j$, we have $$\sum_{\delta_t > 2\ep_j-\delta_j+5} g(\delta_t,\ep_j)(\delta_t-\ep_j) = \sum_{\delta_t > \delta_j+5} g(\delta_t,\ep_j)(\delta_t-\ep_j)+O\left(\frac{\log^3(k)}{\sqrt{k}}\right)\sum_{\delta_t} g(\delta_t,\ep_j)|\delta_t-\ep_j|.$$ Indeed, for example, if $\delta_j < \ep_j$, then {\small$$\sum_{\delta_t > \delta_j+5} g(\delta_t,\ep_j)(\delta_t-\ep_j)-\sum_{\delta_t > 2\ep_j-\delta_j+5} g(\delta_t,\ep_j)(\delta_t-\ep_j)$$ $$= \sum_{\ep_j-(\ep_j-\delta_j)+5 < \delta_t \le \ep_j+(\ep_j-\delta_j)+5} g(\delta_t,\ep_j)(\delta_t-\ep_j)$$ $$\hspace{12mm} = \sum_{\ep_j < \delta_t \le \ep_j + (\ep_j-\delta_j)+5} [g(\delta_t,\ep_j)-g(2\ep_j-\delta_t,\ep_j)](\delta_t-\ep_j)$$ $$\hspace{7mm} = \sum_{\ep_j < \delta_t \le \ep_j+(\ep_j-\delta_j)+5} O\left(\frac{\log^3(k)}{\sqrt{k}}\right)g(\delta_t,\ep_j)(\delta_t-\ep_j).$$}

\noindent Therefore, \eqref{eqn:finalerexplicitcase1} is, up to negligible error, equal to \begin{equation}\label{eqn:finalestexplicitcase1} \sum_{\ep_j}\sum_{\delta_j > \ep_j} \left(\left[f(\ep_j,\delta_j)(\delta_j-\ep_j)+f(\ep_j,2\ep_j-\delta_j)(\ep_j-\delta_j)\right]\sum_{\delta_t > \delta_j+5} g(\delta_t,\ep_j)(\delta_t-\ep_j)\right).\end{equation}

\noindent Lemma 7 then allows us to write \eqref{eqn:finalestexplicitcase1} as $$\sum_{\ep_j} \sum_{\delta_j > \ep_j} \left(\left[(\delta_j-\ep_j)+(\ep_j-\delta_j)\right]f(\ep_j,\delta_j)\sum_{\delta_t > \delta_j+5} g(\delta_t,\ep_j)(\delta_t-\ep_j)\right)$$ up to a negligible error. But this is just $0$, and so we've established \eqref{eqn:main}.

\vsss

\section{Remaining Proofs of Lemmas}

In this section, we prove lemmas 5 and 2, restated here for the reader's convenience.

\setcounter{lemma}{4}
\begin{lemma}
For any fixed positive integers $a,j,t,m,f$ with $|m-2k|,|j-\frac{m}{2}|,|t-\frac{m}{2}|,|f-\frac{2k}{3}|,|a-\frac{f}{2}| \le \sqrt{k}\log(k)$ and $t > j$, the following holds: {\small $$\sum_b {t-j-1 \choose 2b-2a-1}{m-t+1 \choose 2f-2b-1} = \left(\frac{1}{2}+O\left(\frac{\log^2(k)}{t-j}\right)\right) \sum_b {t-j-1 \choose b-2a-1}{m-t+1 \choose 2f-b-1},$$} \noindent where the first sum is restricted to $b$ with $|b-a-\frac{t-j}{3}| \le \sqrt{t-j}\log(k)$, and the second sum is restricted to $b$ with $|b-2a-2\frac{t-j}{3}| \le 2\sqrt{t-j}\log(k)$.
\end{lemma}

\begin{proof}

The sum on the right contains all $b$ in the range $[2a+2\frac{t-j}{3}-2\sqrt{t-j}\log(k),2a+2\frac{t-j}{3}+2\sqrt{t-j}\log(k)]$, while the sum on the left contains only even $b$ in that range. Therefore, due to the factor of $\frac{1}{2}$, we wish to show \eqref{eqn:parity}: {\tiny \begin{equation}\label{eqn:parity} \sum_{b \text{ even}} {t-j-1 \choose 2b-2a-1}{m-t+1 \choose 2f-2b-1} - \sum_{b \text{ odd}} {t-j-1 \choose 2b-2a-1}{m-t+1 \choose 2f-2b-1} = O\left(\frac{\log^2(k)}{t-j}\right)\sum_b {t-j-1 \choose b-2a-1}{m-t+1 \choose 2f-2b-1},\end{equation}} where the range of $b$ is restricted to $|b-2a-2\frac{t-j}{3}| \le 2\sqrt{t-j}\log(k)$. 

\vs

The idea of the proof is to pair every even-$b$ term with $\frac{2}{3}$ times the (odd) term before it and $\frac{1}{3}$ times the (odd) term after it. Specifically, to establish \eqref{eqn:parity}, it suffices to show \eqref{eqn:oddcombo}: {\tiny \begin{equation}\label{eqn:oddcombo} \frac{2}{3}{t-j-1 \choose 2b-2a-1}{m-t+1 \choose 2f-2b+1}+\frac{1}{3}{t-j-1 \choose 2b-2a+1}{m-t+1 \choose 2f-2b-1} = \left(1+O\left(\frac{\log^2(k)}{t-j}\right)\right){t-j-1 \choose 2b-2a}{m-t+1 \choose 2f-2b}.\end{equation}} \hspace{-2.3mm} As mentioned on pages 14-15, the error $O(\frac{\log^2(k)}{\sqrt{t-j}})$ is trivial (it follows from pairing even-$b$ terms with odd-$b$ terms); our $\frac{2}{3}$-$\frac{1}{3}$ weighting gives the (necessary) improvement to $O(\frac{\log^2(k)}{t-j})$. Observe that {\scriptsize$$\frac{2}{3}{t-j-1 \choose 2b-2a-1}{m-t+1 \choose 2f-2b+1}+\frac{1}{3}{t-j-1 \choose 2b-2a+1}{m-t+1 \choose 2f-2b-1}$$} is, by using the equations ${c \choose d-1} = \frac{d}{c-d+1}{c \choose d}$ and ${c \choose d+1} = \frac{c-d}{d+1}{c \choose d}$, equal to {\scriptsize$${t-j-1 \choose 2b-2a}{m-t+1 \choose 2f-2b}$$ $$\indent \times \left[\frac{2}{3}\frac{2b-2a}{t-j-1-(2b-2a)+1}\frac{m-t+1-(2f-2b)}{2f-2b+2}+\frac{1}{3}\frac{t-j-1-(2b-2a)}{2b-2a+1}\frac{2f-2b}{m-t+1-(2f-2b)+1}\right].$$}

\noindent Since $$\frac{m-t+1-(2f-2b)}{2f-2b+2} = \frac{1}{2}+O\left(\frac{\log(k)}{\sqrt{k}}\right)$$ and $$\frac{2f-2b}{m-t+1-(2f-2b)+1} = 2+O\left(\frac{\log(k)}{\sqrt{k}}\right),$$ we may replace the expression above in brackets with, up to an acceptable error, \begin{equation}\label{eqn:combo} \frac{2}{3}\frac{b-a}{t-j-1-(2b-2a)}+\frac{1}{3}\frac{t-j-1-(2b-2a)}{b-a}.\end{equation} Writing $b = a+\frac{t-j-1}{3}+\Delta$ transforms \eqref{eqn:combo} into $$\frac{(\frac{t-j-1}{3})^2+2\Delta^2}{(\frac{t-j-1}{3})^2-\frac{t-j-1}{3}\Delta-2\Delta^2},$$ which is $1+O\left(\frac{\log^2(k)}{t-j}\right)$, the critical point being the lack of a $\frac{t-j-1}{3}\Delta$ term (which is why we chose the factors $\frac{2}{3}$ and $\frac{1}{3}$). This finishes the proof of \eqref{eqn:oddcombo} and thus \eqref{eqn:parity}. 
\end{proof}

\vs

\setcounter{lemma}{1}
\begin{lemma}
For any real $\eta$ bounded away from $0$ and $1$, any positive integers $A$ and $B$ such that $\eta A, \eta B \in \Z$, and any integers $\Delta$ and $\sigma$ such that $A+\Delta,\eta A+\sigma, B-\Delta$, and $\eta B-\sigma$ are non-negative, it holds that {\tiny $$\left[\frac{{A+\Delta \choose \eta A+\sigma}{B-\Delta \choose \eta B-\sigma}}{{A \choose \eta A}{B \choose \eta B}}\right]^{-1} =$$ $$(1+O(\frac{\sigma^3}{A^2}))(1+O(\frac{\Delta^3}{A^2}))(1+O(\frac{1}{A}))(1+O(\frac{\sigma(\Delta-\sigma)^2}{A^2}))(1+O(\frac{\Delta(\Delta-\sigma)^2}{A^2}))\exp\left(\frac{1}{2}\frac{(\Delta-\sigma)^2}{(1-\eta)A}+\frac{1}{2}\frac{\sigma^2}{\eta A}-\frac{1}{2}\frac{\Delta^2}{A}\right)$$ $$\times (1+O(\frac{\sigma^3}{B^2}))(1+O(\frac{\Delta^3}{B^2}))(1+O(\frac{1}{B}))(1+O(\frac{\sigma(\Delta-\sigma)^2}{B^2}))(1+O(\frac{\Delta(\Delta-\sigma)^2}{B^2}))\exp\left(\frac{1}{2}\frac{(\Delta-\sigma)^2}{(1-\eta)B}+\frac{1}{2}\frac{\sigma^2}{\eta B}-\frac{1}{2}\frac{\Delta^2}{B}\right).$$}
\end{lemma}

\begin{proof}
Using Stirling's approximation {\ssmall$$n!=\left(1+O\left(\frac{1}{n}\right)\right)\frac{n^n}{e^n}\sqrt{2\pi n},$$} we obtain {\ssmall $$\left[\frac{{A+\Delta \choose \eta A+\sigma}{B-\Delta \choose \eta B-\sigma}}{{A \choose \eta A}{B \choose \eta b}}\right]^{-1} = (1+O(\frac{1}{A}))(1+O(\frac{1}{B})) \times$$ $$\frac{(\eta A+\sigma)^{\eta A+\sigma}((1-\eta)A+\Delta-\sigma)^{(1-\eta)A+\Delta-\sigma}(\eta B-\sigma)^{\eta B-\sigma}((1-\eta)B-(\Delta-\sigma))^{(1-\eta)B-(\Delta-\sigma)}A^AB^B}{(\eta A)^{\eta A}((1-\eta)A)^{(1-\eta)A}(\eta B)^{\eta B}((1-\eta)B)^{(1-\eta)B}(A+\Delta)^{A+\Delta}(B-\Delta)^{B-\Delta}}$$ $$ = (1+O(\frac{1}{A}))(1+O(\frac{1}{B}))\left[\frac{\eta A+\sigma}{(1-\eta)A+(\Delta-\sigma)}\frac{(1-\eta)B-(\Delta-\sigma)}{\eta B-\sigma}\right]^\sigma$$ $$\hspace{23mm} \times \left[\frac{(1-\eta)A+(\Delta-\sigma)}{A+\Delta}\frac{B-\Delta}{(1-\eta)B-(\Delta-\sigma)}\right]^\Delta (1+\frac{\sigma}{\eta A})^{\eta A}(1+\frac{\Delta-\sigma}{(1-\eta)A})^{(1-\eta)A}$$ $$\indent \hspace{.2mm} \times (1-\frac{\sigma}{\eta B})^{\eta B}(1-\frac{\Delta}{A+\Delta})^A(1-\frac{\Delta-\sigma}{(1-\eta)B})^{(1-\eta)B}(1+\frac{\Delta}{B-\Delta})^B.$$}

\noindent Now, using that $\log(1+x) = x-\frac{x^2}{2}+O(x^3)$ for small $x$, {\ssmall$$(1+\frac{\sigma}{\eta A})^{\eta A}(1+\frac{\Delta-\sigma}{(1-\eta)A})^{(1-\eta)A}(1-\frac{\sigma}{\eta B})^{\eta B}(1-\frac{\Delta}{A+\Delta})^A(1-\frac{\Delta-\sigma}{(1-\eta)B})^{(1-\eta)B}(1+\frac{\Delta}{B-\Delta})^B$$ $$ = \exp\left(\eta A\left(\frac{\sigma}{\eta A}-\frac{1}{2}\frac{\sigma^2}{\eta^2 A^2}+O(\frac{\sigma^3}{A^3})\right)\right)\exp\left((1-\eta)A\left(\frac{\Delta-\sigma}{(1-\eta)A}-\frac{1}{2}\frac{(\Delta-\sigma)^2}{(1-\eta)^2A^2}+O(\frac{(\Delta-\sigma)^3}{A^3})\right)\right)$$ $$\indent \times \exp\left(-A\left(\frac{\Delta}{A+\Delta}+\frac{1}{2}\frac{\Delta^2}{(A+\Delta)^2}+O(\frac{\Delta^3}{(A+\Delta)^3})\right)\right) \exp\left(-\eta B\left(\frac{\sigma}{\eta B}+\frac{1}{2}\frac{\sigma^2}{\eta^2 B^2}+O(\frac{\sigma^3}{B^3})\right)\right) \times $$ $$\exp\left(-(1-\eta)B\left(\frac{\Delta-\sigma}{(1-\eta)B}+\frac{1}{2}\frac{(\Delta-\sigma)^2}{(1-\eta)^2B^2}+O(\frac{(\Delta-\sigma)^3}{B^3})))\exp(B(\frac{\Delta}{B-\Delta}+\frac{1}{2}\frac{\Delta^2}{(B-\Delta)^2}+O(\frac{\Delta^3}{(B-\Delta)^3})\right)\right)$$ $$\hspace{-17mm} = (1+O(\frac{\sigma^3}{A^2}))(1+O(\frac{\Delta^3}{A^2}))(1+O(\frac{(\Delta-\sigma)^3}{A^2}))\exp(-\frac{1}{2}\frac{\sigma^2}{\eta A}-\frac{1}{2}\frac{(\Delta-\sigma)^2}{(1-\eta)A}-\frac{1}{2}\frac{\Delta^2}{A}+\frac{\Delta^2}{A})$$ $$\hspace{-7mm} \times (1+O(\frac{\sigma^3}{B^2}))(1+O(\frac{\Delta^3}{B^2}))(1+O(\frac{(\Delta-\sigma)^3}{B^2}))\exp(-\frac{1}{2}\frac{\sigma^2}{\eta B}-\frac{1}{2}\frac{(\Delta-\sigma)^2}{(1-\eta)B}-\frac{1}{2}\frac{\Delta^2}{B}+\frac{\Delta^2}{B}).$$}

\noindent And using the simpler $\log(1+x) = x+O(x^2)$ for small $x$, {\ssmall$$\left[\frac{\eta A+\sigma}{(1-\eta)A+(\Delta-\sigma)}\frac{(1-\eta)B-(\Delta-\sigma)}{\eta B-\sigma}\right]^\sigma$$ $$\hspace{-24mm} = \left[1+\frac{(1-\eta)\sigma B-(\Delta-\sigma)\eta B+\sigma(1-\eta)A-\eta(\Delta-\sigma)A}{(1-\eta)\eta AB+(\Delta-\sigma)\eta B-\sigma(1-\eta)A-\sigma(\Delta-\sigma)}\right]^\sigma$$ $$\hspace{25mm} = \exp\left(\sigma\left(\frac{\sigma}{\eta A}-\frac{\Delta-\sigma}{(1-\eta)A}+\frac{\sigma}{\eta B}-\frac{\Delta-\sigma}{(1-\eta)B}+O(\frac{\sigma^2}{A^2})+O(\frac{(\Delta-\sigma)^2}{A^2})+O(\frac{\sigma^2}{B^2})+O(\frac{(\Delta-\sigma)^2}{B^2})\right)\right)$$} and {\ssmall $$\left[\frac{(1-\eta)A+(\Delta-\sigma)}{A+\Delta}\frac{B-\Delta}{(1-\eta)B-(\Delta-\sigma)}\right]^\Delta$$ $$\hspace{-10.5mm} = \left[1+\frac{(\Delta-\sigma)B-(1-\eta)\Delta B+(\Delta-\sigma)A-(1-\eta)\Delta A}{(1-\eta)AB+(1-\eta)\Delta B-(\Delta-\sigma)A-\Delta(\Delta-\sigma)}\right]^\Delta$$ $$\hspace{7.5mm} = \exp\left(\Delta\left(\frac{\Delta-\sigma}{(1-\eta)A}-\frac{\Delta}{A}+\frac{\Delta-\sigma}{(1-\eta)B}-\frac{\Delta}{B}+O(\frac{(\Delta-\sigma)^2}{A^2})+O(\frac{\Delta^2}{A^2})+O(\frac{(\Delta-\sigma)^2}{B^2})+O(\frac{\Delta^2}{B^2})\right)\right).$$}

\noindent Combining everything yields the lemma.
\end{proof}

\vspace{1mm}

\section{Large Hamming Distances} 

The lower bounds established for trace reconstruction thus far have come from pairs of strings with small Hamming distance. A natural question is what can be said about strings with very large Hamming distance. Of course, a pair of strings that differ in all but $O(1)$ indices can be distinguished very easily (in $O(1)$ traces). However, what if we insist on ``padding" two strings that always differ, at the beginning and end by some arbitrary strings? 

\vs

We say that a pair of strings $x,y \in \{0,1\}^n$ \textit{essentially always differ} if there are indices $k_1,k_2 \le n$ such that $x$ and $y$ agree at all indices at most $k_1$ and at least $k_2$, and disagree at all indices between $k_1$ and $k_2$. 

\vs

\setcounter{proposition}{7}
\begin{proposition}
Let $x,y \in \{0,1\}^n$ be a pair of strings that essentially always differ. Then $x$ and $y$ can be distinguished in $\exp(C\frac{\log^3 n}{\log\log n})$ samples. Here, $C > 0$ is an absolute constant. 
\end{proposition}

We use the following lemma, found as E7 on page 64 of [13].

\setcounter{lemma}{8}
\begin{lemma}
Let $p(z) = a_nz^n+\dots+a_1z+a_0$ be a polynomial of degree $n$ with $a_i \in \{\pm 1\}$ for each $i$. Then, $p(z)$ has at most $\frac{C\log^2 n}{\log\log n}$ zeros at $1$, i.e., $(z-1)^m$ does not divide $p(z)$ for $m = \lfloor \frac{C\log^2 n}{\log\log n}\rfloor+1$. Here, $C > 0$ is an absolute constant.
\end{lemma}

\vs

With this lemma, we deduce Proposition 8 as follows. We first claim that there is some $0-1$ string $w$ of length at most $k := \lfloor \frac{C\log^2 n}{\log\log n} \rfloor+1$ such that $f(w;x) \not = f(w;y)$ (see Lemma 1 for notation). Indeed, if $f(w;x) = f(w;y)$ for all $w$ of length at most $k$, that is, if the so-called ``$k$-decks" of $x$ and $y$ are the same, then by Section 5 of [11], it must be that $\sum_{i=1}^n x_i i^m = \sum_{i=1}^n y_i i^m$ for all $0 \le m \le k-1$. If we let $p(z) = \sum_{i=1}^n [x_i-y_i]z^i$, then it's easy to see that the equalities imply $p(1),p'(1),\dots,p^{(k-1)}(1) = 0$, which imply $(z-1)^k \mid p(z)$. Now, since $x$ and $y$ essentially always differ, $p(z)$ takes the form $p(z) = \epsilon_{k_1}z^{k_1}+\epsilon_{k_1+1}z^{k_1+1}\dots+\epsilon_{k_2-1}z^{k_2-1}+\epsilon_{k_2}z^{k_2}$ for some $\epsilon_{k_1},\dots,\epsilon_{k_2} \in \{\pm 1\}$. Therefore, by factoring out $z^{k_1}$ and noting $k_2-k_1 \le n$, Lemma 9 implies $k \le \frac{C\log^2 n}{\log\log n}$, a contradiction. The claim is established.

With this claim, we can distinguish between $x$ and $y$ by simply looking at $f(w;\wt{U})$ for traces $\wt{U}$; indeed, $\mathbb{E}_x[f(w;\wt{U})] = f(w;x)(1-q)^{-|w|}$. Since $|w| \le C\frac{\log^2 n}{\log\log n}$, it holds that $\exp(C'\frac{C\log^2 n}{\log\log n}\log n)$ traces suffice to distinguish between $x$ and $y$. For details, see the proof of Theorem 14 of [10].

\section{Acknowledgments}

I would like to thank Omer Tamuz for introducing me to the wonderful trace reconstruction problem, and for helpful discussions. I would also like to thank Russell Lyons for much helpful feedback on the paper, and for a bijective proof of Lemma 1. Finally, I would like to greatly thank an anonymous referee for several helpful comments, substantially improving the paper's readability and understandability. 

\vs

\end{document}